\documentclass[oneside]{amsart}
\usepackage[mathscr]{eucal}
\usepackage{amssymb, amsmath, array, hyperref}
\usepackage{enumerate}
\usepackage{enumitem}

\usepackage{verbatim}

\newfont{\sheaf}{eusm10 scaled\magstep1}

\usepackage{xy}
\xyoption{all}

\def\N{\ensuremath{\mathbb N}}
\def\Z{\ensuremath{\mathbb Z}}

\def\P{\ensuremath{\mathbb P}}

\def\im{\operatorname{im}}
\def\rank{\operatorname{rank}}
\def\dim{\operatorname{dim}}

\def\Hom{{\operatorname {Hom}}}

\newtheorem{THM}{Theorem}[section]

\newtheorem{theorem}{Theorem}
\newtheorem{proposition}[theorem]{Proposition}

\newtheorem{definition}[theorem]{Definition}
\newtheorem{remark}[theorem]{Remark}
\newtheorem{corollary}[theorem]{Corollary}

\numberwithin{equation}{section}

\begin{document}


\title  [Varieties of complexes and foliations]    {Varieties of complexes and foliations}

\author{Fernando Cukierman}

\dedicatory{}

\begin{abstract}
Let $\mathcal F(r, d)$ denote the moduli space of algebraic foliations of codimension one and degree $d$ in complex proyective space of dimension $r$. We show that  $\mathcal F(r, d)$ may be represented as a certain linear section of a variety of complexes. From this fact we obtain information on the irreducible components of $\mathcal F(r, d)$.
\end{abstract}

\thanks{    }


\subjclass{    }

\maketitle

\tableofcontents

\noindent
\section{Basics on varieties of complexes.}
\label{basic}

\noindent

\subsection { }  \label{subsection1}

Let $K$ be a field and let $V_0, \dots, V_n$ be vector spaces over $K$ of finite dimensions
$$d_i = \dim_K (V_i)$$
Consider sequences of linear functions   
$$
\xymatrix {{V_0}  \ar[r]^{f_1} & {V_1} \ar[r]^{f_2} & \dots \ar[r]^{f_n} & {V_n}}
$$
also written
$$f = (f_1, \dots, f_n) \in V = \prod_{i=1}^n \Hom_K(V_{i-1}, V_i)$$
The variety of differential complexes is defined as
$$\mathcal C = \mathcal C(V_0, \dots, V_n) = \{f = (f_1, \dots, f_n) \in V / \ 
f_{i+1} \circ f_i = 0,  \  i = 1, \dots, n -1 \}$$
It is an affine variety in $V$, given as an intersection of quadrics.
We intend to study the geometry of this variety (see also e. g. \cite{Br}, \cite{K}).

\medskip

\subsection { } \label{subsection2}

Since the defining equations $f_{i+1} \circ f_i = 0$ are bilinear, we may also consider, 
when it is convenient, the projective variety of complexes 
$$P \mathcal C \subset  \prod_{i=1}^n \mathbb P \Hom_K(V_{i-1}, V_i)$$
as a subvariety of a product of projective spaces. 

\medskip

Denoting  $V_{\cdot} = \oplus_{i=0}^n V_i$, each complex $f \in \mathcal C$ may be thought as a degree-one homomorphism
of graded vector spaces $f: V_{\cdot} \to V_{\cdot}$ with $f^2 = 0$.

\medskip

\subsection { }  \label{subsection3}

For each  $f \in \mathcal C$ and $i = 0, \dots, n $ define
$$B_i  = f_i(V_{i-1}) \subset Z_i = \text{ker }(f_{i+1}) \subset V_i$$
and
$$H_i = Z_i / B_i$$
(we understand by convention that $B_0 = 0$)

\medskip

From the exact sequences
$$ 0 \to B_i  \to Z_i  \to  H_i  \to 0$$
$$0 \to Z_i  \to V_i  \to B_{i+1} \to 0$$
we obtain for the dimensions 
$$b_i = \dim_K (B_i), \ \ z_i = \dim_K (Z_i), \ \  h_i = \dim_K (H_i)$$
the relations
$$d_i = b_{i+1} + z_i = b_{i+1} + b_i + h_i$$
where $i = 0, \dots, n$ and $b_0 = b_{n+1} = 0$.
Therefore,

\medskip

\begin{proposition} \label{proposition1}

a) The $h_i$ and the $b_j$ determine each other by the formulas:
$$h_i  = d_i -  (b_{i+1} + b_i)$$
$$b_{j+1} = \chi_j(d)  - \chi_j(h)$$
where for a sequence $e = (e_0, \dots, e_n)$ and $0 \le j \le n$ we denote
$$\chi_j(e) = (-1)^{j } \sum_{i=0}^j (-1)^i e_i = e_{j} - e_{j-1} + e_{j-2} + \dots + (-1)^j e_{0}$$
 the $j$-th Euler characteristic of $e$.

\medskip

b) The inequalities $b_{i+1} + b_i  \le d_i$ are satisfied for all $i$.

\end{proposition}

\medskip

\begin{proof}     

\smallskip

We write down the  $b_j$ in terms of the $h_i$: from
$$\sum_{i=0}^j (-1)^i d_i = \sum_{i=0}^j (-1)^i (b_{i+1} + b_i + h_i)$$
we obtain
$$b_{j+1} = (-1)^j ( \sum_{i=0}^j (-1)^i d_i   -  \sum_{i=0}^j (-1)^i h_i )$$
as claimed.

\end{proof} 

\smallskip

Notice in particular that since $b_{n+1} = 0$, we have the usual relation
$$\sum_{i=0}^n (-1)^i d_i   = \sum_{i=0}^n (-1)^i h_i$$

\medskip

\subsection { } \label{subsection4}

Now we consider the subvarieties of $\mathcal C$ obtained by imposing rank conditions on the $f_i$.

\medskip

\begin{definition} \label{definition1} 
For each $r = (r_1, \dots, r_n) \in \mathbb N^n$ define
$$\mathcal C_r = \{f = (f_1, \dots, f_n) \in \mathcal C /  \ 
\rank(f_i) = r_i,  \  i = 1, \dots, n \}$$
\end{definition}

These are locally closed subvarieties of $\mathcal C$.

\medskip

\begin{proposition} \label{proposition2}

\medskip
\noindent
a) $\mathcal C_r \ne \emptyset$  if and only if  $r_{i+1} + r_i  \le d_i$  for  $0 \le i \le n$
(we use the convention $r_0 = r_{n+1} = 0$)

\medskip
\noindent
b) In the conditions of a), $\mathcal C_r$ is smooth and irreducible, of dimension

$$
\dim (\mathcal C_r) = \sum_{i=0}^n  (d_i - r_i)  (r_{i+1} + r_i ) 
= \sum_{i=0}^n  (d_i - r_i) (d_i - h_i) 
= \frac 1 2 \sum_{i=0}^n  (d_{i}^2 - h_{i}^2) 
$$

\end{proposition}

\medskip

\begin{proof}     

\smallskip

a) One implication follows from  Proposition \ref{proposition1}. Conversely, in the given conditions, we want
to construct a complex with $\rank(f_i) = r_i$ for all $i$. Suppose we
constructed 

$$
\xymatrix {{V_0}  \ar[r]^{f_1} & {V_1} \ar[r]^{f_2} & \dots \ar[r]^{f_{n-1}} & {V_{n-1}}}
$$

We need to define $f_{n}: V_{n-1} \to V_{n}$ such that 
$f_n \circ f_{n-1} = 0$ and $\rank(f_n) = r_n$,
that is, a map $V_{n-1}/ B_{n-1} \to V_{n}$ of rank $r_n$.
Such a map exists since $\dim (V_{n-1}/ B_{n-1}) = d_{n-1} -  r_{n-1} \ge  r_n$.

\medskip

b) Consider the projection (forgeting $f_n$)
$$\pi: \mathcal C(V_0, \dots, V_n)_r  \to \mathcal C(V_0, \dots, V_{n-1})_{\bar r}$$
where $r = (r_1, \dots, r_n)$ and $\bar r = (r_1, \dots, r_{n-1})$.
Any fiber $\pi^{-1}(f_1, \dots, f_{n-1})$ is isomorphic to
the subvariety in $\Hom(V_{n-1}/ B_{n-1}, V_{n})$ of maps
of rank $r_n$; therefore, it is smooth and irreducible of dimension
$r_n(d_{n-1} -  r_{n-1} + d_{n} -  r_{n})$ (see \cite{ACGH}).
The assertion follows by induction on $n$. The various expressions for
$\dim (\mathcal C_r)$ follow by direct calculations.

\medskip

Another proof of a): Given $r$ such that $r_{i+1} + r_i  \le d_i$,
put $h_i =  d_i - (r_{i+1} + r_i) \ge 0$ and 
$z_i =  d_i - r_{i+1} = h_i + r_i$. Choose linear subspaces
$B_i   \subset Z_i  \subset V_i$ with $\dim(B_i ) = r_i$
and $\dim(Z_i ) = z_i$. 
Since $\dim(V_{i -1} /Z_{i -1}) = \dim(B_{i})$, choose
an isomorphism $\sigma_i: V_{i -1} /Z_{i -1}  \to  B_{i}$ for each $i$.
Composing with the natural projection $V_{i -1}  \to  V_{i -1} /Z_{i -1}$
we obtain linear maps $V_{i -1}  \to  B_{i}$ with kernel $Z_{i -1}$
and rank $r_i$, as wanted.

 \end{proof}
 
\medskip

\begin{remark} \label{remark1}

In terms of dimension of homology, the condition in Proposition \ref{proposition3} a) translates as follows.
Given $h = (h_0, \dots, h_n) \in \mathbb N^{n+1}$, there exists a complex
with dimension of homology equal to $h$ if and only if
$\chi_i(h) \le \chi_i(d)$ for $i=1, \dots, n-1$ and $\chi_n(h) = \chi_n(d)$.

\end{remark}

\smallskip 
\smallskip

\begin{remark}  \label{remark2}

The group $G = \prod_{i=0}^n \text{ GL}(V_i, K)$
acts on $V = \prod_{i=1}^n \Hom_K(V_{i-1}, V_i)$
via 
$$(g_0, g_1, \dots, g_n) \cdot (f_1, f_2, \dots, f_n) =
(g_0  f_1 g_1^{-1}, g_1  f_2 g_2^{-1}, \dots, g_{n-1}  f_n g_n^{-1})$$
This action clearly preserves the variety of complexes.
It follows from the proof above that the action on each $\mathcal C_r$ is transitive.
Hence, the non-empty $\mathcal C_r$ are the orbits of $G$ acting on 
$\mathcal C(V_0, \dots, V_n)$.

\end{remark}

\medskip

\begin{definition}

For $r, s \in \mathbb N^n$ we write $s \le r$ if $s_i \le r_i$ for $i = 1, \dots, n$.

\end{definition}

\medskip

\begin{corollary} \label{corollary1}

 If $\mathcal C_r \ne \emptyset$ and $s \le r$ then  $\mathcal C_s \ne \emptyset$.
Also, $\dim (\mathcal C_s) > 0$ if $s \ne 0$.

\end{corollary}

\smallskip\smallskip

 \begin{proof}  The first assertion follows from Proposition \ref{proposition2} a), and the second from Proposition \ref{proposition2} b).   \end{proof}

\medskip

\begin{proposition} \label{proposition3}

  With the notation above,
  
$$\overline {\mathcal C}_r = \bigcup_{s \le r} {\mathcal C}_s = 
\{f \in \mathcal C / \ \rank(f_i) \le r_i , \  i = 1, \dots, n \}$$

\end{proposition}

 \begin{proof}
 
 Denote $X_r  = \bigcup_{s \le r} {\mathcal C}_s$.
Since the second equality is clear,  $X_r $ is closed.
It follows that  $\overline {\mathcal C}_r  \subset X_r$. 
To prove the equality, since ${\mathcal C}_r  \subset X_r$ is open,
it would be enough to show that $X_r$ is irreducible.
For this, consider $L = (L_1, \dots, L_n)$ where $L_i \in \text {Grass }(r_i, V_i)$
and denote
$$X_L = \{f = (f_1, \dots, f_n) \in \mathcal C /  \
\text{ im }(f_i) \subset L_i \subset \text{ ker }(f_{i+1}),  \  i = 1, \dots, n \}$$
Consider

$$
\tilde X_r = \{(L, f) / \ f \in X_L\} \  \subset \  G \times \mathcal C
$$

where $G = \prod_{i=0}^n \text {Grass }(r_i, V_i)$. 
The first projection $p_1: \tilde X_r  \to G$ has fibers
$$p_1^{-1}(L) = X_L \cong \Hom(V_0, L_1) \times 
\Hom(V_1/L_1, L_2) \times  \dots 
 \times  \Hom(V_{n-1}/L_{n-1}, V_n)$$
which are vector spaces of constant dimension $\sum_{i=0}^n (d_i - r_i) r_{i+1}$.
It follows that $\tilde X_r$ is irreducible, and hence $X_r = p_2(\tilde X_r)$
is also irreducible, as wanted. 

 \end{proof}

\medskip

\begin{remark}  \label{remark3}

 In the proof above we find again the formula
$$\dim(X_r) = \dim(X_L) + \dim(G) =
\sum_{i=0}^n (d_i - r_i) r_{i} +\sum_{i=0}^n (d_i - r_i) r_{i+1}$$

\end{remark}

\smallskip

\begin{remark}  \label{remark4}

The fact that  $p_1: \tilde X_r  \to G$ is a vector bundle implies that 
$\tilde X_r$ is smooth.
On the other hand, since $p_2 : \tilde X_r \to X_r $ is birational 
(an isomorphism over the open set $\mathcal C_r$), it is a resolution of singularities.

\end{remark}

\medskip

The following two corollaries are immediate consequences of Proposition \ref{proposition3}.

\medskip

\begin{corollary} \label{corollary2} 
${\mathcal C}_s  \subset \overline{\mathcal C}_r$ if and only if $s \le r$.
\end{corollary}

\medskip

\begin{corollary}\label{corollary3}
$\overline {\mathcal C}_r  \cap \overline {\mathcal C}_s  = \overline {\mathcal C}_t$
where $t_i = \text {min }(r_i, s_i)$ for all $i = 1, \dots, n$.
\end{corollary}

\medskip

\begin{definition} \label{} 
For $d = (d_0, \dots, d_n) \in \mathbb N^{n+1}$ let
$$R = R(d) = \{(r_1, \dots, r_{n}) \in \mathbb N^{n}/ \ r_1  \le d_0, \  r_{i+1} + r_i  \le d_i   \  (1 \le i \le n-1),   \ r_n  \le d_n \}$$ 
\end{definition}

\medskip
\noindent
We consider $\mathbb N^{n}$ ordered via $r \le s$ if $r_i \le s_i$ for all $i$; the finite set $R$ has the induced order. Notice that $R$ is finite since it is contained in the box $\{(r_1, \dots, r_{n}) \in \mathbb N^{n}/ \ 0 \le r_i  \le d_i, \  i = 1, \dots, n\}$. 

\medskip

\begin{proposition}  \label{proposition4}
With the notation above, the irreducible components of the variety of complexes
$\mathcal C = \mathcal C(V_0, \dots, V_n)$ are the $\overline {\mathcal C}_r$ with
$r \in R(d_0, \dots, d_n)$ a maximal element.

\end{proposition}

 \begin{proof}
 
 From the previous Propositions, we have the equalities
$$\mathcal C = \bigcup_{r \in R} \mathcal C_r = \bigcup_{r \in R} \overline {\mathcal C}_r =
\bigcup_{r \in R^+} \overline {\mathcal C}_r$$
where $R^+$ denotes the set of maximal elements of $ R $. The result follows because
we know that each
$ \overline{\mathcal C}_r  $ is irreducible and there are no inclusion relations 
among the $ \overline{\mathcal C}_r  $ for $ r \in R^+ $ (see  Corollary \ref{corollary2}). 

 \end{proof}
 

\medskip

\subsection {Morphisms of complexes. Tangent space of the variety of complexes.}  \label{subsection7} Now we would like to compute the dimension of the tangent space of a variety of complexes at each point.

\medskip
\noindent
With the notation of \ref{subsection1} we consider complexes $f \in \mathcal C(V_0, \dots, V_n)$ and $f' \in \mathcal C(V'_0, \dots, V'_n)$ (the vector spaces $V_i$ and $V'_i$ are not necessarily the same, but the lenght $n$ we may assume is the same). We denote
$$\Hom_{\mathcal C}(f, f')$$
the set of morphisms of complexes from $f$ to $f'$, that is, collections of linear maps $g_i: V_i \to V'_i$ for $i=0, \dots, n$, such that $g_i \circ f_i = f'_i \circ g_{i-1}$  for 
$i=1, \dots, n$. It is a vector subspace of $\prod_{i=0}^n \Hom_K(V_{i}, V'_i)$, and we would like to calculate its dimension.

\medskip
\noindent
For this particular purpose and for its independent interest, we recall the following from \cite{B} ($\S 2 - 5.$ Complexes scind\'es): 

\medskip
\noindent
For $f \in \mathcal C(V_0, \dots, V_n)$, denote as in \ref{subsection1}
$$B_i(f)  = f_i(V_{i-1}) \subset Z_i(f) = \text{ker }(f_{i+1}) \subset V_i$$
Since we are working with vector spaces, we may choose linear subspaces $\bar B_i$ and $\bar H_i$ of $V_i$ such that
$$V_i = Z_i(f) \oplus \bar B_i  \ \ \  \text{and }   \ \ \    Z_i(f) = B_i(f) \oplus  \bar H_i$$
Then $V_i = B_i(f) \oplus  \bar H_i  \oplus \bar B_i$ and clearly
$f_{i+1}$ takes $\bar B_i$ isomorphically onto $B_{i+1}(f)$. Notice also that 
$$\dim (\bar B_i) = \dim (B_{i+1}(f)) = \rank(f_{i+1}) = r_{i+1}(f)$$
and
$$ \dim (\bar H_i) =  \dim (Z_i(f) / B_i(f)) = h_i(f)$$
\noindent
Next, define the following complexes:

\medskip
\noindent
$\bar H(i)$ the complex of lenght zero consisting of the vector space $\bar H_i$ in degree $i$, the vector space zero in degrees $\ne i$, and all differentials equal to zero.

\medskip
\noindent
$\bar B(i)$ the complex of lenght one consisting of the vector space $\bar B_{i-1}$ in degree $i-1$, the vector space $B_i(f)$ in degree $i$, with the map 
$f_i: \bar B_{i-1} \to B_i(f)$, and zeroes everywhere else.

\medskip

\begin{proposition}  \label{proposition5}
With the notation just introduced, $\bar H(i)$ and $\bar B(i)$ are subcomplexes of $f$ and  we have a direct sum decomposition of complexes:
$$ f =  \bigoplus_{0 \le i \le n} \bar H(i)  \ \oplus \  \bigoplus_{0 \le i \le n} \bar B(i)$$
\end{proposition}

\begin{proof}
 Clear from the discussion above; see also \cite{B}, loc. cit.
\end{proof}

\medskip
\medskip
\noindent
Now we are ready for the calculation of $\dim_K \Hom_{\mathcal C}(f, f')$. 

\medskip

\begin{proposition}  \label{dimension hom}
With the previous notation, we have:

\begin{eqnarray*}
\dim_K \Hom_{\mathcal C}(f, f') &=& \sum_i h_i h'_i + h_i r'_i + r_i h'_{i-1} + r_i r'_i + r_i r'_{i-1} \\
 &=& \sum_i h_i (h'_i + r'_i) + r_i d'_{i-1} 
\end{eqnarray*}

\end{proposition}

\begin{proof}
 
We may decompose $f$ and $f'$ as in Proposition \ref{proposition5}:

\begin{eqnarray*}
\Hom_{\mathcal C}(f, f') & = & \Hom_{\mathcal C}(\oplus_i \bar H(i)  \oplus \oplus_i \bar B(i), \oplus_i \bar H(i)'  \oplus \oplus_i \bar B(i)') \\
 & = & \oplus_{i,j} \Hom_{\mathcal C}(\bar H(i), \bar H(j)') \ \oplus \ \oplus_{i,j} \Hom_{\mathcal C}(\bar H(i), \bar B(j)')  \oplus  \\
&  & \oplus_{i,j} \Hom_{\mathcal C}(\bar B(i), \bar H(j)') \ \oplus \ \oplus_{i,j} \Hom_{\mathcal C}(\bar B(i), \bar B(j)')
\end{eqnarray*}

\noindent
It is easy to check the following:

\begin{eqnarray*}
\Hom_{\mathcal C}(\bar H(i), \bar H(j)') &=& 0  \  \  \text{for } i \ne j \\
\Hom_{\mathcal C}(\bar H(i), \bar H(i)') &=& \Hom_{K}(\bar H_i, \bar H'_i) 
\end{eqnarray*}

\begin{eqnarray*}
\Hom_{\mathcal C}(\bar H(i), \bar B(j)') &=& 0  \  \  \text{for } i \ne j \\
\Hom_{\mathcal C}(\bar H(i), \bar B(i)') &=& \Hom_{K}(\bar H_i, \bar B'_i) 
\end{eqnarray*}

(the case $j=i+1$ requires special attention)

\begin{eqnarray*}
\Hom_{\mathcal C}(\bar B(i), \bar H(j)') &=& 0  \  \  \text{for } i-1 \ne j \\
\Hom_{\mathcal C}(\bar B(i), \bar H(i-1)') &=& \Hom_{K}(\bar B_{i-1}, \bar H'_{i-1}) \cong \Hom_{K}(\bar B_i(f), \bar H'_{i-1})
\end{eqnarray*}

(the case $j=i$ requires special attention)

\begin{eqnarray*}
\Hom_{\mathcal C}(\bar B(i), \bar B(i)') & \cong & \Hom_{K}(B_i(f), B'_i(f)) \\
\Hom_{\mathcal C}(\bar B(i), \bar B(i-1)') &=& \Hom_{K}(\bar B_{i-1}, B'_{i-1})  \cong \Hom_{K}(B_i(f), B'_{i-1}) \\
\Hom_{\mathcal C}(\bar B(i), \bar B(j)') &=& 0  \  \  \text{otherwise } 
\end{eqnarray*}

\medskip

Taking dimensions we obtain the stated formula.

\end{proof}

\noindent
Now we deduce the dimension of the tangent space to a variety of complexes at any point.

\medskip

\begin{proposition}  \label{tangent space}
For $f \in \mathcal C = \mathcal C(V_0, \dots, V_n)$ we have a canonical isomorphism
$$T\mathcal C(f) = \Hom_{\mathcal C}(f, f(1)) $$
where $T\mathcal C(f)$ is the Zariski tangent space to $\mathcal C$ at the point $f$, and $f(1)$ denotes de shifted complex 
$f(1)_i = (-1)^i f_{i+1}, \ \ i = -1, 0, \dots, n$. 
\end{proposition}

\begin{proof}
Since $\mathcal C$ is an algebraic subvariety of the vector space $V = \prod_{i=1}^n \Hom_K(V_{i-1}, V_i)$, an element of 
$T\mathcal C(f)$ is a $g = (g_1, \dots, g_n) \in V$ such that $f + \epsilon g$ satisfies the equations defining $\mathcal C$ (i. e. a $K[\epsilon]$-valued point of 
$\mathcal C$), that is,
$$(f + \epsilon g)_{i+1} \circ (f + \epsilon g)_i = 0,  \  \ i = 1, \dots, n -1 $$
which is equivalent to
$$f_{i+1} \circ g_i  +  g_{i+1} \circ f_i  = 0,  \  \  i = 1, \dots, n -1 $$
and this means precisely that $g \in \Hom_{\mathcal C}(f, f(1))$.
\end{proof}

\medskip

\begin{corollary}  \label{dim tangent space}
For $f \in \mathcal C = \mathcal C(V_0, \dots, V_n)$,
\begin{eqnarray*}
\dim_K T\mathcal C(f) &=&  \sum_i h_i (h_{i+1} + r_{i+1}) + r_i d_{i} \\
&=&  \sum_i (d_i - r_{i} - r_{i+1}) (d_{i+1} - r_{i+2}) + r_i d_{i} 
\end{eqnarray*}
\end{corollary}

\begin{proof}
From Proposition  \ref{tangent space} we know that $\dim_K T\mathcal C(f) = \dim_K \Hom_{\mathcal C}(f, f(1))$.
Next we apply Proposition \ref{dimension hom} with $f' = f(1)$, that is, replacing $d'_i = d_{i+1}$, $r'_i = r_{i+1}$, $h'_i = h_{i+1}$,  to obtain the result.

\end{proof}


\subsection {Varieties of exact complexes.}  \label{subsection8}

Now we apply the previous results to the case of exact complexes.

\medskip

Let us fix $(d_0, \dots, d_n) \in \mathbb N^n$ so that

$$
\chi_j(d) = (-1)^j  \sum_{i=0}^j (-1)^i d_i \ge 0, \  \ j=1, \dots, n-1 $$

$$
\chi_n(d) = (-1)^n  \sum_{i=0}^n (-1)^i d_i = 0
$$

Denoting $ \chi = \chi(d) = (\chi_1(d), \dots, \chi_n(d)) \in \mathbb N^n$, let us
consider the variety $ \mathcal C_{\chi} $ of complexes of rank $ \chi $ as in Definition \ref{definition1} .
Since $ \chi_{i}(d)  + \chi_{i+1}(d) = d_i $ for all $ i $, it follows from Proposition \ref{proposition2} that
$ \mathcal C_{\chi} $ is non-empty of dimension 
$$\frac 1 2 \sum_{i=0}^n  d_i ^2 $$
It follows from Proposition \ref{proposition1}  that any complex $ f \in \mathcal C_{\chi} $ is exact.
Also, since $ \chi \in R $ is clearly maximal (see Proposition \ref{proposition4}), 
$ \overline {\mathcal C}_{\chi} $ is an irreducible component of $ \mathcal C $.
Let us denote
$$ \mathcal E = \mathcal E(d_0, \dots, d_n) = \overline {\mathcal C}_{\chi} = 
\{f \in \mathcal C / \ \rank(f_i) \le \chi_i ,  \ \  i = 1, \dots, n \} $$
the closure of the variety $ {\mathcal C}_{\chi} $ of exact complexes. 
Denote also, for $ i=1, \dots, n $
$$ \chi^i = \chi - e_i = (\chi_1, \dots, \chi_{i-1}, \chi_{i} -1, \chi_{i+1}, \dots, \chi_n) $$
and
$$ \Delta_i = \overline {\mathcal C}_{\chi^i} = 
\{f \in \mathcal C / \ \rank(f) \le \chi - e_i \} $$
the variety of complexes where the $ i$-th matrix drops rank by one.
\medskip
\begin{proposition} \label{proposition6} The codimension of $ \Delta_i $ in $ \mathcal E $ is equal to one, and
$$ \mathcal E = {\mathcal C}_{\chi} \cup \Delta_1 \cup \dots \cup \Delta_n$$
\end{proposition}
\smallskip
 \begin{proof}
 This follows from   Proposition \ref{proposition3} and the fact that $ s \in \mathbb N^n $ satisfies
$ s <  \chi $ if and only if $ s \le \chi - e_i $ for some $ i=1, \dots, n$.
 \end{proof}
 
 \newpage
  
\noindent
\section{Moduli space of foliations.} \label{foliations}
\noindent
\subsection { }  \label{subsection8}
Let $X$ denote a (smooth, complete) algebraic variety over the complex numbers, let $L$ be a line bundle on $X$ and let $\omega$ denote a global section of 
$\Omega^1_X \otimes L$ (a twisted differential 1-form). A simple local calculation shows that $\omega \wedge d\omega$ is a section of  $\Omega^3_X \otimes L^{\otimes 2}$. We say that $\omega$ is integrable if it satisfies the Frobenius condition  $\omega \wedge d\omega = 0$. We denote 
$$\mathcal F(X, L) \subset \mathbb P H^0(X, \Omega^1_X \otimes L)$$
the projective classes of integrable 1-forms. The map 
$$\varphi: H^0(X, \Omega^1_X \otimes L) \to H^0(X, \Omega^3_X \otimes L^{\otimes 2})$$
 such that  $\varphi(\omega) = \omega \wedge d\omega$ is a homogeneous quadratic map between vector spaces and hence $\varphi^{-1}(0) = \mathcal F(X, L)$ is an algebraic variety defined by homogeneous quadratic equations.

\medskip

Our purpose is to understand the geometry of  $\mathcal F(X, L)$. In particular, we are interested in the problem of describing its irreducible components. For a survey on this problem see for example \cite{LN}.

\subsection { }  \label{subsection9}
\medskip
Let $r$ and $d$ be natural numbers. Consider a differential 1-form in $\mathbb C^{r+1}$
$$\omega = \sum_{i=0}^r a_i  dx_i$$
where the $a_i$ are homogeneous polynomials of degree $d - 1$ in variables $x_0, \dots, x_r$, with complex coefficients. We say that $\omega$ has degree $d$ (in particular the 1-forms $dx_i$ have degree one). Denoting $R$ the radial vector field, let us assume that
$$ <\omega, R> = \sum_{i=0}^r a_i  x_i = 0$$
so that $\omega$ descends to the complex projective space $\mathbb P^r$ as a global section of the twisted sheaf of 1-forms
$\Omega^1_{\mathbb P^r}(d)$.
\medskip
We denote
$$\mathcal F(r, d) =  \mathcal F(\mathbb P^r, \mathcal O(d))$$
parametrizing 1-forms  of degree $d$ on $\mathbb P^r$  that satisfy the Frobenius integrability condition.

\newpage

\noindent
\section{Complexes associated to an integrable form.} \label{complex}
\medskip
Let us denote 
$$H^0(\mathbb P^r, \Omega^k_{\mathbb P^r}(d)) =  \Omega^k_{r}(d)$$
and 
$$\Omega_r = \bigoplus_{d \in \mathbb Z}  \bigoplus_{k = 0, \dots, r}   \Omega^k_{r}(d) $$
with structure of bi-graded commutative associative algebra given by exterior product $\wedge$ of differential forms. 
\medskip
\begin{definition}
Gelfand, Kapranov and Zelevinsky defined in \cite{GKZ} another product in $\Omega_r$, the second multiplication $*$, as follows:
$$\omega_1 * \omega_2 =  \frac {d_1}{d_1 + d_2} \omega_1 \wedge d\omega_2 + (-1)^{(k_1+1)(k_2+1)} \frac {d_2}{d_1 + d_2} \omega_2 \wedge d\omega_1$$
where $\omega_i \in  \Omega^{k_i}_{r}(d_i)$ for $i= 1, 2$. 
\end{definition}
In particular, if $\omega_1$ is a 1-form ($k_1 = 1$) then 
$$\omega_1 * \omega_2 =  \frac {d_1}{d_1 + d_2} \omega_1 \wedge d\omega_2 + \frac {d_2}{d_1 + d_2} \omega_2 \wedge d\omega_1$$
\medskip
\begin{remark}  \label{remark5} For $\omega_i \in  \Omega^{k_i}_{r}(d_i)$ for $i= 1, 2$ as above,

\smallskip

a)  $\omega_1 * \omega_2$ belongs to $\Omega^{(k_1+ k_2+1)}_{r}(d_1 + d_2)$

\smallskip

b)  $\omega_1 * \omega_2 = (-1)^{(k_1+1)(k_2+1)} \omega_2 * \omega_1$. 

\smallskip

c) It follows from an easy direct calculation that $*$ is associative (see \cite{GKZ}).

\smallskip

d) For any $\omega \in \Omega^1_{r}(d)$ we have $\omega * \omega =   \omega \wedge d\omega$. In particular, $\omega$ is integrable if and only if $\omega * \omega = 0$.
\end{remark}
\medskip
\begin{definition} \label{definition4}
For  $\omega \in \Omega^k_{r}(d)$ we consider the operator $\delta_{\omega}$ 
$$\delta_{\omega}: \Omega_r \to \Omega_r$$
such that $\delta_{\omega}(\eta) = \omega * \eta$ for $\eta \in \Omega_r$. 
\end{definition} 
\medskip

\begin{remark}  \label{remark6}  From Remark \ref{remark5} a), if $\omega \in \Omega^{k_1}_{r}(d_1)$ then 
$$\delta_{\omega}(\Omega^{k_2}_{r}(d_2)) \subset\Omega^{(k_1+ k_2+1)}_{r}(d_1 + d_2)$$
In particular, if $\omega \in \Omega^{1}_{r}(d_1)$, 
$$\delta_{\omega}(\Omega^{k_2}_{r}(d_2)) \subset\Omega^{(k_2+2)}_{r}(d_1 + d_2)$$
\end{remark}

\medskip

\begin{corollary}  \label{corollary4}
$\omega \in \Omega^1_{r}(d)$ is integrable if and only if $\delta_{\omega}^2 = 0$
\end{corollary}

\smallskip

\begin{proof}  The associativity stated in Remark \ref{remark5} c) implies that 
$\delta_{\omega_1} \circ \delta_{\omega_2} = \delta_{\omega_1 * \omega_2}$. In particular, 
$\delta_{\omega}^2 = \delta_{\omega * \omega}$ and hence the claim follows from  Remark \ref{remark5} d).
 \end{proof}

\medskip
\begin{definition} \label{definition5}
For  $\omega \in \Omega^1_{r}(d)$ and $e \in \Z$ we define two differential graded vector spaces
$$C^+_{\omega}(e): \Omega^{0}_{r}(e) \to \Omega^{2}_{r}(e+d) \to \Omega^{4}_{r}(e+2d) \to \dots \to  \Omega^{2k}_{r}(e+kd) \to \dots$$
$$C^-_{\omega}(e): \Omega^{1}_{r}(e) \to \Omega^{3}_{r}(e+d) \to \Omega^{5}_{r}(e+2d) \to \dots \to  \Omega^{2k+1}_{r}(e+kd) \to \dots$$
where all maps are $\delta_{\omega}$ as in Remark \ref{remark6}.
\end{definition} 
\medskip

\begin{remark}  \label{remark7} 
It follows from Corollary \ref{corollary4} that   $C^+_{\omega}(e)$ and $C^-_{\omega}(e)$ are differential complexes (for any $e \in \Z$) if and only if $\omega$ is integrable.
\end{remark}
\medskip

\begin{remark}  \label{remark8} 
To fix ideas we shall mostly discuss   $C^-_{\omega}(e)$, but similar considerations will apply to $C^+_{\omega}(e)$. If no confusion seems to arise we shall denote $C^-_{\omega}(e) = C_{\omega}(e)$.
\end{remark}

\medskip
\begin{proposition} \label{proposition7} 
Let $\omega \in \Omega^1_{r}(d)$, $e \in \Z$ and $k \in \N$ such that $k+2 \le r$. Then $\omega * \eta = 0$ for all $\eta \in  \Omega^k_r(e)$ if and only if 
$\omega = 0$. In other words, the linear map
$$ \delta:  \Omega^1_{r}(d) \to \Hom_K(\Omega^{k}_{r}(e), \Omega^{k+2}_{r}(e+d))$$
sending $\omega \mapsto \delta_{\omega}$, is injective.
\end{proposition}
\smallskip
\begin{proof} First remark that $\omega \wedge \eta = 0$ for all $\eta \in  \Omega^k_r(e)$ (with $k+1 \le r$) easily implies $\omega = 0$. 
Now suppose  $\omega * \eta = 0$, that is, $d \ \omega \wedge d\eta + e \  \eta \wedge d\omega = 0$,  for all $\eta \in \Omega^k_r(e)$. 
Take $\eta = x_{i_1}^{e-k} dx_{i_1} \wedge \dots \wedge dx_{i_k}$ (here $x_i$ denote affine coordinates and $1 < i_1 < \dots i_k < n$).  Since $d\eta = 0$, we have 
$dx_{i_1} \wedge \dots \wedge dx_{i_k} \wedge d\omega = 0$. Hence $d\omega = 0$ by the first remark. Using the hypothesis again, we know $\omega \wedge d\eta = 0$  for all $\eta \in  \Omega^k_r(e)$. Now take
$\eta = x_{i_{k+1}}^{e-k} dx_{i_1} \wedge \dots \wedge dx_{i_k}$ (where $1 < i_1 < \dots < i_{k+1} < n$). It follows that 
$dx_{i_1} \wedge \dots \wedge dx_{i_{k+1}} \wedge \omega = 0$ and hence $\omega = 0$.
 \end{proof}

\medskip
\begin{theorem} \label{theorem1} Fix  $e \in \Z$. Let us consider the graded vector space
$$\Omega_{r}(e) = \bigoplus_{0 \le k \le [\frac{r-1}{2}]} \Omega^{2k+1}_{r}(e+kd)$$
(direct sum of the spaces appearing in $C^-_{\omega}(e)$ above). Define the linear map 
$$\delta(e) = \delta:  \Omega^1_{r}(d) \to  \prod_{k=1}^{[\frac{r-1}{2}]} \Hom_K(\Omega^{2k-1}_{r}(e+(k-1)d), \Omega^{2k+1}_{r}(e+kd)) $$
such that $ \delta(\omega) = \delta_{\omega}$ for each $\omega \in \Omega^1_{r}(d)$, and its projectivization
$$\P \delta: \P  \Omega^1_{r}(d) \to  \prod_{k=1}^{[\frac{r-1}{2}]} \P \Hom_K(\Omega^{2k-1}_{r}(e+(k-1)d), \Omega^{2k+1}_{r}(e+kd)) $$
Denote $\mathcal C = \mathcal C( \Omega^{1}_{r}(e), \Omega^{3}_{r}(e+d), \Omega^{5}_{r}(e+2d), \dots, \Omega^{2[\frac{r-1}{2}]+1}_{r}(e+[\frac{r-1}{2}]d))$ the variety of complexes as in \ref{subsection1} and  $\mathcal F(r, d)$ the variety of foliations as in \ref{subsection9}. Then
$$\mathcal F(r, d) = (\P \delta) ^{-1}(\mathcal C)$$
In other terms, $\P \delta(\mathcal F(r, d)) = L \cap \mathcal C$, that is, the variety of foliations $\mathcal F(r, d)$ corresponds via the linear
map $\P \delta$ to the intersection of the variety of complexes with the linear space $L = \im (\P \delta)$.

\end{theorem}
\smallskip
 \begin{proof}
 The statement is a rephrasing of Corollary \ref{corollary4} or Remark  \ref{remark7}.
 \end{proof}

\medskip
\begin{proposition} \label{proposition8} 
Let us denote
$$d^k_r(e) = \dim \Omega^{k}_{r}(e) = \binom{r-k+e}{r-k} \binom{d-1}{k}$$
(see \cite{OSS}) and in particular
$$d_k = d^{2k+1}_r(e+kd) = \dim \Omega^{2k+1}_{r}(e+kd), \  0 \le k \le [\frac{r-1}{2}]$$
For this $d = (d_0, d_1, \dots, d_{[\frac{r-1}{2}]} )$ we consider the finite ordered set $R = R(d)$ as in Proposition \ref{proposition4}. 
Then each irreducible component of the variety of foliations 
$\mathcal F(r, d)$ is an irreducible component of the linear section $(\P \delta)^{-1}(\overline {\mathcal C}_r)$ for a unique $r \in R^+$. 

\end{proposition}

\smallskip
 \begin{proof}
From  Proposition \ref{proposition4}, we have the decomposition into irreducible components
$$\mathcal C = \bigcup_{r \in R^+} \overline {\mathcal C}_r$$
From Theorem \ref{theorem1} we obtain: 
$$\mathcal F(r, d) = (\P \delta) ^{-1}(\mathcal C) = \bigcup_{r \in R^+}  (\P \delta) ^{-1}(\overline {\mathcal C}_r) $$
and this implies that each irreducible component $X$ of $\mathcal F(r, d)$ is an irreducible component of $(\P \delta)^{-1}(\overline {\mathcal C}_r)$  for some $r \in R^+$. 
This element $r$ is the sequence of ranks of $\delta_{\omega}$ for a general $\omega \in X$, hence it is unique.
 \end{proof}

\vspace {8 mm}
\small
\begin{flushleft}
  Universidad de Buenos Aires / CONICET \\
  Departamento de Matem\'{a}tica, FCEN  \\
  Ciudad Universitaria  \\
  (1428) Buenos Aires  \\
  ARGENTINA \\
  fcukier@dm.uba.ar 
\end{flushleft}
\end{document}